\documentclass[12pt]{amsart}
\usepackage[utf8]{inputenc}
\usepackage{amssymb}
\usepackage[frenchstyle,largesmallcaps,fulloldstylenums,nott,easyscsl]{kpfonts}

\calclayout

\usepackage{enumitem}
\setlist[enumerate]{labelindent=--.5in,leftmargin=20pt}
\setlist[itemize]{labelindent=--.5in,leftmargin=13pt}

\usepackage[colorlinks=true]{hyperref}
\usepackage{color}
\definecolor{green}{RGB}{0,127,0}
\definecolor{red}{RGB}{191,0,0}

\usepackage{parskip}
\newcommand{\B}[1]{\mathbf {#1}}
\newcommand{\C}[1]{\mathcal {#1}}
\newcommand\OP{\operatorname}

\newtheorem{theorem}{Theorem}

\newtheorem{lemma}[theorem]{Lemma}

\theoremstyle{definition}
\newtheorem{definition}[theorem]{Definition}
\newtheorem{question}{Question}

\title[On approximation properties of semidirect products]{On approximation properties of semidirect products of groups}

\author[G. Arzhantseva]{Goulnara Arzhantseva}
\address{Universit\"at Wien\\
Fakult\"at f\"ur Mathematik\\
Oskar--Morgenstern--Platz~1\\ 
1090 Wien, Austria}
\thanks{The research of both authors was supported in part by the Swiss \textsc{\MakeLowercase{NSF}},
under Sinergia grant \textsc{\MakeLowercase{CRSI22--130435}}. The research of the first author was
supported in part by her \textsc{\MakeLowercase{ERC}} grant \textsc{\MakeLowercase{ANALYTIC}} no. 259527.}
\email{goulnara.arzhantseva@univie.ac.at}

\author{Światosław R. Gal}
\address{Uniwersytet Wrocławski\\
Instytut Matematyczny\\
pl.~Grunwaldzki~2/4\\
50--384 Wrocław, Poland}
\thanks{The research of the second author was
supported in part by Polish National Science Center (\textsc{\MakeLowercase{NCN}}) grants
\textsc{\MakeLowercase{2012/06/A/ST1/00259}} and \textsc{\MakeLowercase{2017/27/B/ST1/01467}}.}
\email{sgal@math.uni.wroc.pl}


\begin{document}

\begin{abstract}
Let  $\C R$ be a class of groups closed under taking (split) extensions with
finite kernel and fully residually $\C R$--groups.  We prove that  $\C R$
contains all (split) $\{$~finitely generated residually finite~$\}$--by--$\C R$ groups.
It follows that a split extension with a 
finitely generated residually finite
kernel and a surjunctive quotient  is surjunctive. This remained unknown even for
direct products of a surjunctive group with the integers~$\B Z$. 

\bigskip
\centerline{\textit{\textbf{{Sur les propriétés d'approximation des produits semi-directs des groupes}}}}
\medskip\noindent
\textsc{Rrésumé}: Soit $\C R$ une classe de groupes fermée par rapport aux extensions (scin\-dées) avec
un noyau fini et par rapport aux groupes multi-résiduellement $\C R$.  Nous montrons que $\C R$
contient toutes les extensions (scindées) de type $\{$~finiment engendré résidu\-ellement fini~$\}$--par--$\C R.$ 
Nous obtenons en corollaire qu'une extension scindée avec un noyau finiment engendré
résiduellement fini et un quotient surjonctif est surjonctive. Cela restait inconnu, même pour les
produits directs d'un groupe surjonctif avec les entiers~$\B Z$. 
\end{abstract}

\maketitle

The concept of approximation is one of the most fundamental in science.  In
nowadays geometric group theory it refers to algebraic approximations such as
residual finiteness, to approximations in the space of marked groups such as
local embeddability into a given class of groups,
e.g.~\cite{MR1458419,MR2151593}, or to metric approximations such as soficity
and hyperlinearity \cite{MR2460675,MR3408561}.

A current intensive study of sofic and hyperlinear groups, motivated by a
variety of deep results on these wide -- and still mysterious: is there a non
sofic/hyperlinear group? -- classes of groups, has given rise to a number of
questions on group-theoretical properties both of these recently discovered
groups and of their fundamental predecessors. One of such a major question is
whether or not a given class of groups is preserved under taking (split or, in other words,
semidirect) extensions.

Every semidirect extension of a residually finite group with a finitely
generated residually finite kernel is residually finite by an elegant result of
Mal'cev \cite{Malcev}. A more general is a class of groups locally embeddable
into finite ones (briefly, \textsc{lef}-groups).  There exist examples of semidirect
extensions of \textsc{lef}-groups which are not \textsc{lef} \cite{MR1458419}.  With regard to
metric approximations, the only known and positive results on extensions are on
(not necessarily split) extensions by an amenable group \cite{MR2460675,MR3408561}.

Our main result is the following general theorem which can be then applied to
many groups, including all those just mentioned, and, in particular, to
surjunctive groups (see the definition below), which came to light after
Gromov's spectacular proof of the surjunctivity of all sofic
groups \cite{MR1694588}.

\begin{theorem}\label{thm:main}
Assume that $\C R$ is a class of groups such that
\begin{itemize}
\item fully residually $\C R$--groups belong to $\C R$ and
\item (split) extension of $\C R$--groups with finite kernel belong to $\C R$.
\end{itemize}
Let $G$ be a (split) extension with a finitely generated
residually finite kernel $K$ and a quotient group $Q$ in $\C R$:
$$
1 \to K\hookrightarrow G \twoheadrightarrow Q\to 1.
$$
  Then $G$ belongs to $\C R$.
\label{t:sdp}
\end{theorem}

\begin{proof}
Take a finite subset $S\subseteq G$. Define $F:=S^{-1}S$ and choose a finite
index subgroup $K_0\trianglelefteq K$ such that $F\cap K_0=\{e\}$.  Since $K_0$
is of finite index and $K$ is finitely generated one can find a finite index subgroup
$K_1\leqslant K_0$ which is characteristic, that is, $\OP{Aut}(K)$--invariant.
Then, for $N:=K/K_1$ (note that, by construction,  $K_1$ is normal in $G$, and hence, in $K$), the
group $G_1:=G/K_1$ has the following properties:
\begin{enumerate}
\item $S$ injects into $G_1;$\label{p:inj}
\item $G_1$ is a (split) extension of $Q$ with finite kernel $N$.\label{p:ext}
\end{enumerate} 
Denote the quotient map by $\pi\colon
G\twoheadrightarrow G_1$.  Let $s_1,s_2\in S$.  If $\pi(s_1)=\pi(s_2),$ then
$\pi(s_1^{-1}s_2)=e$.  However, $s_1^{-1}s_2\in S^{-1}S\cap\OP{ker}(\pi) =
F\cap K_1\subseteq F\cap K_0=\{e\}$.  Thus, $s_1=s_2$ and we obtain (\ref{p:inj}).  For
(\ref{p:ext}), we observe that the quotient $G_1/N$ is isomorphic to $G/K\cong Q$ and $N$ is finite by construction. In addition, a splitting homomorphism $G/K\to G$ induces a splitting homomorphism $G_1/N\cong (G/K_1) / (K/K_1) \to G/K_1$.

By (\ref{p:ext}) and the hypothesis on $\C R$,   $G_1$ belongs to $\C R$.  As a consequence, since $S$ is an arbitrary
finite subset of $G$, the group $G$ is fully residually $\C R$. Thus, $G$
belongs to $\C R$.
\end{proof}

\begin{lemma}
Assume that $\C R$ is a class of groups such that
\begin{itemize}
\item direct products of a finite and of an $\C R$--group belong to $\C R$ and
\item all finite index subgroups of $\C R$--groups belong to $\C R$.
\end{itemize}
Then any semidirect extension with finite kernel of an $\C R$--group belongs to
$\C R$.\label{l:fin--ext}
\end{lemma}

\begin{proof}
Let $G$ be a split extension with finite kernel of an $\C R$--group.  Then $G$ admits a retraction $\varphi \colon G\to H$
on its finite index subgroup which belongs to $\C R$.  Consider the homomorphism $G\hookrightarrow
\OP{Sym}(G/H)\times H\colon g\mapsto (fH\mapsto gfH, \varphi(g))$.  It is injective and the image has finite index in
$\OP{Sym}(G/H)\times H$ which belongs to $\C R$.  Thus, $G$ belongs to $\C R$.
\end{proof}

Let us apply our theorem to groups mentioned above. We begin with surjunctive
groups.

The concept of surjunctivity was introduced by Gottschalk \cite{MR0407821}  in
topological dynamics in 1973. Applied to actions of a discrete group $G$, it
can be viewed as an analogue of co--Hopf property of groups and of Artinian
modules.

\begin{definition}
A continuous dynamical system $(X,G)$ is called {\em surjunctive} if every
continuous injective map $f\colon X\to X$ commuting with the action of~$G$ is
surjective.
\end{definition}

\begin{definition}
Given a finite alphabet $\Sigma$ the associated {\em Bernoulli shift} is the
space~$\Sigma^G$ of~$\Sigma$--valued functions on~$G$.
\end{definition}

\begin{definition}
A group~$G$ is called {\em surjunctive} if the Bernoulli shift $\Sigma^G$ is
surjunctive {\em for any} finite alphabet~$\Sigma$.
\end{definition}

All sofic groups are surjunctive \cite{MR1694588,MR1803462} and the class of
sofic groups is the largest known class of surjunctive groups. It is still
unknown whether every group is surjunctive and whether there exists a non-sofic
surjunctive group.  In particular, it is not clear whether there exists a group
$G$ such that $\{0,1\}^G$ is surjunctive but $\{0,1,2\}^G$ is~not.

A very few affirmative results are known regarding the class of surjunctive groups:
\begin{itemize}
\item Finite groups are surjunctive \cite[Proposition 3.1.3]{MR2683112}.
\item Every subgroup of a surjunctive group is surjunctive
\cite[Proposition 3.2.1]{MR2683112}.
\item Locally surjunctive groups are surjunctive \cite[Proposition 3.2.2]{MR2683112}.
In particular, $S_\infty=\bigcup\limits_{n\to\infty}S_n$, the group of finitely
supported permutations of $\B N$, is surjunctive.
\item Fully residually surjunctive groups are surjunctive
\cite[Lemma 3.3.4]{MR2683112}. In particular, Abelian groups and free groups
are surjunctive.
\end{itemize}

\begin{lemma}
Virtually surjunctive groups are surjunctive.  In particular, semidirect
extension with finite kernel of a surjunctive group is surjunctive.
\label{l:virt--surj}
\end{lemma}

\begin{proof}
Assume that $H\leqslant G$ is of finite index and $H$ is surjunctive.  Then we
have $\Sigma^G=\left(\Sigma^{G/H}\right)^H$, as $H$--spaces, and every
$G$--invariant map is $H$--invariant.  This yields the claim.
\end{proof}

It is unknown whether a product of two surjunctive groups is
surjunctive \cite{MR1803462}.  The question is open both for direct and free
products. A particular case of free products with group of integers~$\B Z$ is
open as well.

The concept of surjunctivity extends to a more broad context where the
space~$\Sigma^G$ of~$\Sigma$--valued functions on~$G$ is considered for
$\Sigma$ being an object in a {\em surjunctive category}~$\C C$.  This
naturally leads to the notion of {\em $\C C$-surjunctive
group} \cite{ChSilbCoorn}.  For instance, the above definition of surjunctive
groups corresponds to $\C C$ being the category of finite sets and if $\C C$ is
the category of finite-dimensional vector spaces over an arbitrary field one
obtains the notion of linear surjunctivity \cite[Chapter 8.14]{MR2683112}.

Without going into technical details (precise definitions are given in
\cite{ChSilbCoorn}) a surjunctive category is a category where every injective
$\C C$-endomorphism (self-map) is surjective.  Let us give examples of
surjunctive categories (see \cite[Example 7.3]{ChSilbCoorn} and references therein):
\begin{itemize}
\item the category of finite sets,
\item the category of finite-dimensional vector spaces over an arbitrary field,
\item the category of left Artinian modules over an arbitrary ring, of finitely
generated left modules over an arbitrary left-Arti\-ni\-an ring,
\item the category of affine algebraic sets over an arbitrary uncountable
algebraically closed field, or
\item the category of compact topological manifolds.
\end{itemize}

Let us note that all the categories above are closed under taking finite direct
products (cf. property (\textsc{\MakeLowercase{FP}}) in \cite[Section 3.1]{ChSilbCoorn}).  

Our
Lemma~\ref{l:virt--surj} remains valid for the class of $\C C$-surjunctive
groups with $\C C$ any  surjunctive category closed under taking finite direct products;
in particular, for any $\C C$ from the list above.

We extend Malcev's result who considered the case of $\C R$ being the class of
residually finite groups \cite{Malcev}.

\begin{theorem}
Let $\C R$ denote one of the following classes of groups:
\begin{itemize}
\item residually amenable groups,
\item groups locally embeddable into finite groups (that is, \textsc{lef}--groups),
\item initially subamenable groups (in other words, \textsc{lea}--groups),
\item sofic groups,
\item surjunctive groups, or
\item $\C C$-surjunctive groups, where $\C C$ is a surjunctive category closed
under taking finite direct products.
\end{itemize}
Then any semidirect extension of a group from class $\C R$ with residually
finite, finitely generated kernel belongs to $\C R$.
\end{theorem}

\begin{proof}
Fully residually surjunctive groups are surjunctive by \cite[Lemma 3.3.4]{MR2683112}. A straightforward generalization of the proof of \cite[Lemma 3.3.4]{MR2683112} shows that fully residually $\C C$-surjunctive
groups are $\C C$-surjunctive. Analogous statements for other classes are obvious.

To check if an arbitrary semidirect extension with finite kernel of an $\C R$--group belongs to
$\C R$, we apply Lemma~\ref{l:fin--ext} in case of residually amenable,
\textsc{lef}, or \textsc{lea} groups, Lemma~\ref{l:virt--surj} in case of
surjunctive groups and also, as mentioned above, in case of $\C C$-surjunctive groups, as well as \cite[Proposition 7.5.14]{MR2683112} in case of
sofic groups. 

Thus, all the discussed classes satisfy the hypothesis of
Lemma \ref{l:fin--ext} and, therefore, that of the split case of Theorem \ref{t:sdp}.
\end{proof}

\begin{question}\label{q:fbs}
Let $G$ be a non-split extension with a finite kernel and a surjunctive (resp. $\C C$-surjunctive) quotient.  Is $G$ surjunctive (resp. $\C C$-surjunctive)?
\end{question}
A positive answer to this question, combined with the non-split case of Theorem~\ref{thm:main}, will imply that all such extensions with residually
finite, finitely generated kernel are surjunctive (resp. $\C C$-surjunctive). An analogous question for residually amenable,
\textsc{lef}, or \textsc{lea} groups has a negative answer (since there exist extensions with a finite kernel and a residually finite quotient which have Kazhdan's property (T), and therefore are not \textsc{lea}, e.g.~\cite{MR507760, MR528238}).  It remains open
for sofic and hyperlinear groups, c.f.~\cite[Question 4.1]{MR3934790}.
In a particular instance when the kernel is a finite cyclic group, a positive answer to  Question \ref{q:fbs} would imply the surjunctivity of Deligne's non-residually finite central extensions \cite{MR507760}, of its $SL_n$ analogues \cite{MR528238} and of its recent arithmetic analogues \cite{MR3887220}.

\subsection*{Acknowledgments}
The authors thank Benji Weiss for a helpful conversation and for pointing out that
the surjunctivity question was open even
for direct products of a surjunctive group with $\B Z$.

\bibliography{sdp}

\def\cprime{$'$}
\begin{thebibliography}{ABFSG19}

\bibitem[\textsc{\MakeLowercase{ABFSG19}}]{MR3934790}
{\sc Goulnara Arzhantseva, Federico Berlai, Martin Finn-Sell, and Lev Glebsky}.
\newblock Unrestricted wreath products and sofic groups.
\newblock {\em Internat. J. Algebra Comput.}, 29(2):343--355, 2019.

\bibitem[\textsc{\MakeLowercase{CG05}}]{MR2151593}
{\sc Christophe Champetier and Vincent Guirardel}.
\newblock Limit groups as limits of free groups.
\newblock {\em Israel J. Math.}, 146:1--75, 2005.

\bibitem[\textsc{\MakeLowercase{CL15}}]{MR3408561}
{\sc Valerio Capraro and Martino Lupini}.
\newblock {\em Introduction to sofic and hyperlinear groups and {C}onnes'
  embedding conjecture}, volume 2136 of {\em Lecture Notes in Mathematics}.
\newblock Springer, Cham, 2015.
\newblock With an appendix by Vladimir Pestov.

\bibitem[\textsc{\MakeLowercase{CSC10}}]{MR2683112}
{\sc Tullio Ceccherini-Silberstein and Michel Coornaert}.
\newblock {\em Cellular automata and groups}.
\newblock Springer Monographs in Mathematics. Springer-Verlag, Berlin, 2010.

\bibitem[\textsc{\MakeLowercase{CSC13}}]{ChSilbCoorn}
{\sc Tullio Ceccherini-Silberstein and Michel Coornaert}.
\newblock Surjunctivity and reversibility of cellular automata over concrete
  categories.
\newblock In {\sc Massimo~A. Picardello}, editor, {\em Trends in Harmonic
  Analysis}, volume~3 of {\em Springer INdAM Series}, pages 91--133. Springer
  Milan, 2013.

\bibitem[\textsc{\MakeLowercase{Del78}}]{MR507760}
{\sc Pierre Deligne}.
\newblock Extensions centrales non r\'{e}siduellement finies de groupes
  arithm\'{e}tiques.
\newblock {\em C. R. Acad. Sci. Paris S\'{e}r. A-B}, 287(4):A203--A208, 1978.

\bibitem[\textsc{\MakeLowercase{Got73}}]{MR0407821}
{\sc Walter Gottschalk}.
\newblock Some general dynamical notions.
\newblock In {\em Recent advances in topological dynamics ({P}roc. {C}onf.
  {T}opological {D}ynamics, {Y}ale {U}niv., {N}ew {H}aven, {C}onn., 1972; in
  honor of {G}ustav {A}rnold {H}edlund)}, pages 120--125. Lecture Notes in
  Math., Vol. 318. Springer, Berlin, 1973.

\bibitem[\textsc{\MakeLowercase{Gro99}}]{MR1694588}
{\sc Mikhael Gromov}.
\newblock Endomorphisms of symbolic algebraic varieties.
\newblock {\em J. Eur. Math. Soc. (JEMS)}, 1(2):109--197, 1999.

\bibitem[\textsc{\MakeLowercase{Hil19}}]{MR3887220}
{\sc Richard~M. Hill}.
\newblock Non-residually finite extensions of arithmetic groups.
\newblock {\em Res. Number Theory}, 5(1):Art. 2, 27, 2019.

\bibitem[\textsc{\MakeLowercase{Mal56}}]{Malcev}
{\sc Anatoli\u{\i} Mal{\cprime}cev}.
\newblock On homomorphisms onto finite groups.
\newblock {\em Ivanov. Gos. Ped. Inst. U\v c. Zap. Fiz.-Mat. Fak.}, 18:49--60,
  1956.

\bibitem[\textsc{\MakeLowercase{Mil79}}]{MR528238}
{\sc John~J. Millson}.
\newblock Real vector bundles with discrete structure group.
\newblock {\em Topology}, 18(1):83--89, 1979.

\bibitem[\textsc{\MakeLowercase{Pes08}}]{MR2460675}
{\sc Vladimir~G. Pestov}.
\newblock Hyperlinear and sofic groups: a brief guide.
\newblock {\em Bull. Symbolic Logic}, 14(4):449--480, 2008.

\bibitem[\textsc{\MakeLowercase{VG97}}]{MR1458419}
{\sc Anatoli\u{\i}~M. Vershik and Evgeni\u{\i}~I. Gordon}.
\newblock Groups that are locally embeddable in the class of finite groups.
\newblock {\em Algebra i Analiz}, 9(1):71--97, 1997.

\bibitem[\textsc{\MakeLowercase{Wei00}}]{MR1803462}
{\sc Benjamin Weiss}.
\newblock Sofic groups and dynamical systems.
\newblock {\em Sankhy\=a Ser. A}, 62(3):350--359, 2000.
\newblock Ergodic theory and harmonic analysis (Mumbai, 1999).

\end{thebibliography}
\bibliographystyle{sc}

\end{document}